\documentclass[12pt,reqno]{amsart}
\usepackage{latexsym,amsmath,amssymb, amsthm, mathscinet}
\usepackage{cases, verbatim}
\usepackage{graphicx}
\usepackage{bmpsize}
\usepackage{float}

\usepackage{amsfonts,amsmath,amsthm, amssymb}
\usepackage{cases}
\usepackage[usenames]{color}
\usepackage{enumerate}
\usepackage{bm}

\theoremstyle{plain}
\newtheorem{theorem}{Theorem}[section]

\newtheorem{lemma}[theorem]{Lemma}

\newtheorem{prop}[theorem]{Proposition}

\numberwithin{equation}{section}

\newcommand{\N}{\mathbb{N}}

\newcommand{\R}{\mathbb{R}}

\newcommand{\T}{\mathbb{T}}
\newcommand{\Z}{\mathbb{Z}}

\newcommand{\Lip}{{\rm Lip\,}}

\newcommand{\ep}{\varepsilon}

\newcommand{\ol}{\overline}

\begin{document}
\baselineskip=15pt

\title[Quantitative homogenization of first-order ODEs]{Quantitative homogenization of first-order ODEs}
\author{Panrui Ni}
\address[P. Ni]{
Department 1: Department of Mathematics, Faculty of Science and Engineering, Waseda University, 3-4-1 Okubo, Shinjuku-ku, Tokyo, 169-8555, Japan; Department 2: Shanghai Center for Mathematical Sciences, Fudan University, Shanghai 200438, China}
\email{panruini@gmail.com}

\makeatletter
\@namedef{subjclassname@2020}{\textup{2020} Mathematics Subject Classification}
\makeatother

\date{\today}


\begin{abstract}

This paper investigates the quantitative homogenization of first-order ODEs. For single-scale scalar ODEs, we obtain a sharp $O(\ep)$ convergence rate and characterize the effective constant. In the multi-scale setting, our results match those of \cite{IM} for long times but improve the short-time error to $O(\ep)$. We also obtain what appears to be the first uniform $O(\ep)$
homogenization estimate for scalar ODEs with quasi-periodic fast
dependence on the state variable. The scalar framework is further extended to higher dimensions under a bounded-mean-motion-type assumption. For weakly coupled systems with fast switching rates, we derive
an explicit $O(\ep)$ convergence rate that captures the initial layer. These results have applications to linear transport equations and broader connections to PDEs and gradient systems.

\end{abstract}

\maketitle


\section{Introduction}

In this paper, we study the quantitative homogenization of scalar and vector-valued ODEs of the form
\begin{equation*}
\begin{cases}
\dot u^\ep=f\left(\frac{u^\ep}{\ep},\frac{t}{\ep},u^\ep,t\right),\quad t>0,\\ 
u^\ep(0)=c\in\R^n,
\end{cases}
\end{equation*}
where $\ep>0$ is a small parameter and the function $f$ exhibits periodic or quasi-periodic structures in its fast variables. The origins of the study of the homogenization of such equations can be traced back to the seminal work \cite{P1} in the 1970s. Our main goal is to revisit this problem by incorporating ideas recently developed in the study of first-order Hamilton--Jacobi equations, particularly the subadditive method and a cutting technique analogous to the “curve surgery” approach. This perspective yields sharp convergence rates and an infimum-type formula for the effective constants. It also uncovers connections to dynamical systems, including rotation numbers and torus flows.

To demonstrate the scope of this framework, we begin with a core sequence of models of increasing complexity:
\begin{itemize}
\item single-scale periodic scalar ODEs,
\item multi-scale scalar ODEs, previously studied in \cite{IM},
\item higher-dimensional single-scale periodic ODEs (e.g., irrational shear flows).
\end{itemize}
This sequence reflects a gradual development in both structure and dimension. In addition, we consider quasi-periodic scalar ODEs and weakly coupled systems of ODEs with fast switching rates. These problems do not fit into the subadditive structure and are addressed using different arguments. As applications, our results yield quantitative homogenization for linear transport equations. In Section \ref{4.2}, we also outline connections to several applications, such as Hamilton--Jacobi equations arising in dislocation dynamics, semilinear heat equations, free boundary problems, and gradient systems with wiggly energies motivated by physical models.

\subsection{Single-scale scalar ODEs}
We begin with the homogenization of the equation
\begin{equation}\label{e1}
\begin{cases}
\dot u^\ep=f(\frac{u^\ep}{\ep},\frac{t}{\ep}),\quad t>0,\\ u^\ep(0)=c\in\R,
\end{cases}
\end{equation}
where $\ep>0$ and $\dot u^\ep=\frac{du^\ep}{dt}$. This problem admits a complete characterization based on an underlying subadditive structure. We assume
\begin{itemize}
\item [{\bf (f1)}] $f\in C(\R^2)$ is $\kappa$-Lipschitz continuous, $f(r+l,\tau+k)=f(r,\tau)$ for all $(l,k)\in\mathbb Z^2$.
\end{itemize}
By the basic theory of ODEs, the Lipschitz continuity of $f$ ensures the existence and uniqueness of the solution 
$u^\ep(t;c)$ to \eqref{e1} for each $c\in\R$. Since $f$ is periodic, it is bounded, i.e., $\|f\|_\infty<\infty$. As a result, the solution $u^\ep(t;c)$ exists globally for all $t>0$.

Let $v(\tau;c)$ be the solution of
\begin{equation*}
\begin{cases}
\dot v=f(v,\tau),\quad \tau>0,\\ v(0)=c.
\end{cases}
\end{equation*}
It is straightforward to verify that $u^\ep(t;c)=\ep v(\frac{t}{\ep};\frac{c}{\ep})$. 
\begin{theorem}\label{thm1}
Assume {\rm (f1)} holds. Then there is a constant $\bar f\in\R$ such that, for all $c\in\R$, the following estimate holds
\begin{equation}\label{bd}
|u^\ep(t;c)-(c+\bar f t)|\leq (1+2\|f\|_\infty)\ep,\quad \forall t>0.
\end{equation}
Here, $\bar f$ is given by
\begin{equation}\label{infv}
\bar f=\lim_{k\to\infty}\frac{v(k;c)}{k}=\inf_{k>0}\frac{v(k;0)+2(\|f\|_\infty+1)}{k},\quad \forall c\in\R.
\end{equation}
Moreover, if $f(r,\tau)$ is independent of $\tau$, the estimate improves to
\begin{equation}\label{1ep}
|u^\ep(t;c)-(c+\bar f t)|\leq \ep.
\end{equation}
\end{theorem}

The estimate \eqref{bd} was previously established in \cite[Proposition 2.1]{IM}. Different from that work, the approach developed in this paper is based on a combination of the almost-subadditivity of the map $(\tau,c)\mapsto v(\tau;c)$ and its oscillation control, which enables the derivation of explicit convergence rates. In particular, the improved estimate \eqref{1ep} has prefactor $1$,
independent of the data of the equation. The argument is inspired by
\cite{MNT}, where the authors studied the homogenization of
Hamilton--Jacobi equations with $u/\ep$-periodic Hamiltonians. Compared with that setting, the present argument is technically simpler. We do not need to cut curves in the $x$-space, where $x$ denotes the spatial variable in the Hamilton--Jacobi equations, see the proof of Lemma \ref{lem1} below. This simplification allows us to compute explicitly the multiplicative constant in front of $\ep$. In particular, to obtain the convergence rate $O(\ep)$ for \eqref{e1}, it is not necessary to establish the superadditivity of $(\tau,c)\mapsto v(\tau;c)$. In fact, using superadditivity typically yields worse error constants in estimates such as \eqref{bd}. After adding a fixed constant, the relevant quantity becomes
subadditive, and the limit in \eqref{infv} follows from Fekete's lemma. Moreover, the map $c\mapsto v(1;c)$ induces a circle homeomorphism. Therefore, the existence of $\bar f$ provides an alternative proof of the existence of a rotation number for circle homeomorphisms. This result is originally due to the classical work \cite{KH}. Related subadditive approaches to rotation numbers can be found in
\cite{LiLu,JST}.

By applying the method of characteristics, Theorem \ref{thm1} yields an $O(\ep)$ convergence rate for the homogenization of the linear transport equation, see Theorem \ref{thmt} below. This demonstrates that there are subtle connections between circle homeomorphisms, the homogenization of ODEs, and the homogenization of transport equations.

To examine the sharpness of the $O(\ep)$ rate, we recall the following example from \cite[Example 1.6]{IM}. Consider the function $f(r,\tau)=g(r+\tau)-1$, where $g(w)$ is a 1-periodic function equal to $|w-1/2|$ on $[0,1]$. In this case, $\bar f=-1$, and for any $\delta>0$, 
\[u^\ep(t;0)-\bar f t\asymp \ep\quad \text{for}\quad t=\delta \ep|\log\ep|.\]

\subsection{Multi-scale scalar ODEs}

Now we consider the unique solution $u^\ep(t;c)$ of
\begin{equation}\label{fmul}
\begin{cases}
\dot{u}^\ep=f(\frac{u^\ep}{\ep},\frac{t}{\ep},u^\ep,t),\quad t>0,\\ u^\ep(0)=c\in\R.
\end{cases}
\end{equation}
Here, we assume
\begin{itemize}
\item [{\bf (f2)}] 
$f\in C(\R^4)$ is bounded and $\kappa$-Lipschitz continuous. Moreover, $f(r,\tau,u,t)$ is 1-periodic in both $r$ and $\tau$, and non-increasing in $u$.
\end{itemize}

For each $u_0,t_0\in\R$, we consider the unique solution $u^\ep(t;c,u_0,t_0)$ of
\begin{equation}\label{ute}
\begin{cases}
\dot u^\ep=f(\frac{u^\ep}{\ep},\frac{t}{\ep},u_0,t_0),\quad t>0,\\ u^\ep(0)=c\in\R.
\end{cases}
\end{equation}
Then
\[\lim_{\ep\to 0} u^\ep(t;c,u_0,t_0)=c+\bar f(u_0,t_0)t.\]
From the above convergence process, one can define a function $\bar f:\R^2\to\R$. By the comparison principle, if $u_1\geq u_2$, then $u^\ep(t;c,u_1,t_0)\leq u^\ep(t;c,u_2,t_0)$, which immediately implies that $\bar f(u,t)$ is non-increasing in $u$. In fact, the monotonicity of $f$ in $u$ is crucial for the uniqueness of the limiting equation \eqref{ebar}, since, as we will show in the estimate \eqref{modu}, $\bar f$ may not be Lipschitz continuous. Therefore, uniqueness cannot be guaranteed through standard arguments relying on Lipschitz continuity.

\begin{theorem}\label{thm2}
Assume {\em (f2)}. Then $\bar f(u,t)$ admits a modulus of continuity. More precisely, there exists $d_0>0$ such that, for any $t_1,t_2>0$ and
$u_1,u_2\in\R$ satisfying
\[
0<|u_1-u_2|+|t_1-t_2|\le d_0,
\]
we have
\begin{equation}\label{modu}
|\bar f(u_1,t_1)-\bar f(u_2,t_2)|\leq \frac{2}{\Big\lfloor \frac{1}{\kappa}W\Big(\frac{1}{|u_1-u_2|+|t_1-t_2|}\Big)\Big\rfloor},
\end{equation}
where $W$ denotes the inverse of $xe^x$, and $\lfloor x\rfloor$ denotes the floor of $x$, i.e., the greatest integer less than or equal to $x$. Moreover, there is a unique solution $\overline{u}(t;c)$ of
\begin{equation}\label{ebar}
\begin{cases}
\dot{\overline{u}}=\bar f(\overline{u},t),\quad t>0,\\ \overline{u}(0)=c,
\end{cases}
\end{equation}
and a constant $C>0$, depending only on $\|f\|_{W^{1,\infty}(\R^4)}$, such that
\begin{equation}\label{t>d}
\big|u^\ep(t;c)-\overline{u}(t;c)\big|\leq \frac{Ct}{|\log \ep|},\quad t>\ep|\log \ep|,
\end{equation}
and
\begin{equation}\label{t<d}
\big|u^\ep(t;c)-\overline{u}(t;c)\big|\leq \min\big\{C\ep,2\|f\|_\infty t\big\},\quad t\in[0,\ep|\log \ep|].
\end{equation}
\end{theorem}

\medskip

So far, we have focused on the periodic homogenization of scalar first-order ODEs. While there is some overlap with the results in \cite{IM}, our proofs
are independent of those in \cite{IM}. Although an equivalent form of \eqref{bd} was previously established in \cite[Proposition 2.1]{IM}, the corresponding convergence rate was not explicitly stated there. Our argument additionally yields the infimum characterization
\eqref{infv} and the unit-prefactor estimate \eqref{1ep}. Our estimate \eqref{t>d} agrees with \cite[Theorem 1.5]{IM}, but for short time intervals $t\in[0,\ep|\log \ep|]$, we further obtain a sharper estimate \eqref{t<d}, which does not appear in \cite{IM}. The modulus estimate \eqref{modu} is slightly different from \cite[Proposition 3.2]{IM}, being more explicit and depending solely on the Lipschitz constant of $f$, though it does not improve the regularity of $\bar f$.

\subsection{Quasi-periodic scalar ODEs}

While the periodic case has been extensively studied, the quasi-periodic setting remains largely unexplored in the context of homogenization for ODEs and linear transport equations. Quasi-periodic homogenization for ODEs has been studied previously in
settings with oscillatory dependence on the time variable; see, for
example, \cite{Steffe}. Here, by contrast, the quasi-periodicity enters
through the fast state variable $u^\ep/\ep$. To the best of the author's knowledge, Theorem \ref{qp} provides the
first uniform $O(\ep)$ homogenization estimate for scalar ODEs with
quasi-periodic fast dependence on the state variable. Theorem
\ref{thmt}(2-2) gives the corresponding quantitative estimate for
linear transport equations.

We make the following assumption:
\begin{itemize}
\item [{\bf (f3)}] $N\geq 2$ and $\xi_0\in\R^N$ is a Diophantine vector: there exist $C_{\xi_0}>0$ and $\sigma_{\xi_0}>0$ such that \[ |\xi_0\cdot k| \geq C_{\xi_0}|k|^{-\sigma_{\xi_0}}, \quad \text{for all } k\in\Z^N\setminus\{0\}. \] Moreover, $F\in C^m(\mathbb T^N)$ with $m>\sigma_{\xi_0}$, $F>0$ and $f(r)=F(\xi_0 r)$ for $r\in\R$.
\end{itemize}
Then $f\in C^m(\R)$ is quasi-periodic and satisfies the Lipschitz estimate
\[|f(r_1)-f(r_2)|\leq \|DF\|_{\infty} |\xi_0||r_1-r_2|,\quad \forall r_1,r_2\in\R.\]
Now we consider the unique solution $u^\ep(t;c)$ of
\begin{equation}\label{he}
\begin{cases}
\dot{u}^\ep=f(\frac{u^\ep}{\ep}),\quad t>0,\\ u^\ep(0)=c\in\R.
\end{cases}
\end{equation}
\begin{theorem}\label{qp}
Assume {\rm (f3)} holds. Then there exists a constant $C>0$ depending on $N$, $m$, $F$, and the Diophantine data of $\xi_0$, such that, for all $t,\ep>0$ and all $c\in\R$, we have
\[|u^\ep(t;c)-(c+M_f t)|\leq C\ep,\]
where the effective speed is \[ M_f := \left( \int_{\T^N}\frac{1}{F(\xi)}\,d\xi \right)^{-1}. \]
\end{theorem}

\subsection{Higher dimensional case}

Let $n\in\N$ with $n>1$. We consider the unique solution $u^\ep(t;c)$ of the following $\R^n$-valued ODE
\begin{equation}\label{uen}
\begin{cases}
\dot u^\ep=f(\frac{u^\ep}{\ep},\frac{t}{\ep}),\quad t>0,\\ u^\ep(0)=c\in\R^n.
\end{cases}
\end{equation}
We make the following assumption:
\begin{itemize}
\item [{\bf (f4)}] $f\in C(\R^{n+1},\R^n)$ is $\kappa$-Lipschitz continuous, $\mathbb Z^{n+1}$-periodic in $(r,\tau)$.\end{itemize}
Let $v(\tau;c)$ be the solution of the rescaled equation
\begin{equation}\label{ven}
\begin{cases}
\dot v=f(v,\tau),\quad \tau>0,\\ v(0)=c\in\R^n.
\end{cases}
\end{equation}
It is straightforward to verify that $u^\ep(t;c)=\ep v(\frac{t}{\ep};\frac{c}{\ep})$. In what follows, we denote by $u_i$ the $i$-th component of the vector $u\in\R^n$. We further assume
\begin{itemize}
\item [{\bf (f5)}] there is $C_0>0$ such that for all $c^1,c^2\in\R^n$ with $c^1-c^2\in Y:=[0,1]^n$, we have
\begin{equation}\label{C0}
v_i(\tau;c^2)-v_i(\tau;c^1)\leq C_0,\quad \forall i\in\{1,\dots,n\},\quad \forall \tau>0.
\end{equation}
\end{itemize}
\begin{theorem}\label{thmequ}
Assume {\rm (f4)} and {\rm (f5)}. Then for all $c^1,c^2\in\R^n$ with $c^1-c^2\in Y$, we have
\begin{equation}\label{bdtn}
v_i(\tau;c^1)-v_i(\tau;c^2)\leq 1+C_0,\quad \forall i\in\{1,\dots,n\},\ \forall \tau>0.
\end{equation}
For all $c\in\R^n$, there exists a constant $\bar f_i\in\R$ such that
\begin{equation}\label{bm}
|v_i(\tau;c)-(c_i+\bar f_i\tau)|\leq 1+C_0+2\|f_i\|_\infty,\quad \forall i\in\{1,\dots,n\},\ \forall \tau>0,
\end{equation}
where
\begin{equation}\label{infvn}
\bar f_i=\lim_{k\to\infty}\frac{v_i(k;c)}{k}=\inf_{k>0}\frac{v_i(k;0)+2(\|f_i\|_\infty+C_0+1)}{k},\quad \forall c\in\R^n.
\end{equation}
Moreover, if $f_i(r,\tau)$ is independent of $\tau$ for all
$i=1,\ldots,n$, then
\[|v_i(\tau;c)-(c_i+\bar f_i \tau)|\leq 1+C_0.\]
\end{theorem}

As a consequence, we obtain the following homogenization result.
\begin{theorem}\label{thm1n}
Assume {\rm (f4)} and {\rm (f5)}. Then there exists a constant vector $\bar f\in\R^n$ such that for all $t>0$ and all $c\in\R^n$, we have
\[|u^\ep_i(t;c)-(c_i+\bar f_i t)|\leq (1+C_0+2\|f_i\|_\infty)\ep,\]
where $\bar f_i$ is given by \eqref{infvn}. Moreover, if $f_i(r,\tau)$ is independent of $\tau$ for all
$i=1,\ldots,n$, then
\[|u^\ep_i(t;c)-(c_i+\bar f_i t)|\leq (1+C_0)\ep.\]
\end{theorem}

When $n=1$, conditions \eqref{C0} and \eqref{bdtn} hold automatically, due to the order-preserving property of scalar ODEs. Indeed, for $0< c^1-c^2\leq 1$, the comparison principle yields
\[v(\tau;c^2)<v(\tau;c^1)\leq v(\tau;c^2+1)=v(\tau;c^2)+1,\quad \forall \tau>0.\]
The estimate \eqref{bm} is commonly referred to as the bounded mean motion property, see \cite{J1}. Closely related uniform
estimates, called fine asymptotic expansions of ODE flows, were
studied in \cite{BH23}. In the context of flows $\Phi_t$ on $\T^n$, with a lift $\tilde\Phi_t$ defined on $\mathbb R^n$, the rotation vector
\[\lim_{t\to \infty}\frac{\tilde\Phi_t(z)-z}{t},\quad z\in\R^n\]
may fail to exist in general when $n>1$, and even if it exists, it can depend on the initial point $z$. In such cases, the dynamics is better described by the rotation set, a compact convex subset of $\R^n$, rather than a single vector. If the above limit exists and is independent of $z$, $\Phi_t$ is called a pseudo-rotation. Actually, \eqref{infvn} implies that if condition (f5) holds, the induced flow $c \mapsto v(\tau; c)$ defines a pseudo-rotation. Without a uniform rotation vector, one generally cannot expect bounded mean motion for all orbits. We do not claim that condition {\rm (f5)} is necessary
for homogenization in general. Rather, it is used here as a 
sufficient condition for bounded mean motion. In \cite{J2}, the author investigated conservative pseudo-rotations satisfying the bounded mean motion property. In particular, when the rotation vector is totally irrational, $\Phi_t$ is semi-conjugated to an irrational rotation. This can be viewed as a higher-dimensional analogue of Poincar\'e's classification of circle homeomorphisms. A more classical approach to understanding such behavior in higher dimensions is through Kolmogorov-Arnold-Moser (KAM) theory. However, KAM theory requires the flow to be a small perturbation of a linear flow and relies on strong smoothness and Diophantine conditions.

Theorem \ref{thmequ} provides a quantitative criterion equivalent to the bounded mean motion property. In fact, assuming the bounded mean motion property holds, condition (f5) follows directly from a triangle inequality estimate. Therefore, to verify bounded mean motion, it suffices to check the upper bound in \eqref{C0}.  While condition (f5) may appear restrictive in general settings, we show in Proposition \ref{prop:shear} below that \eqref{C0} holds for irrational shear flows. As a consequence, we obtain an alternative proof of \cite[Theorem 2 (a)]{KS}, as an application of Theorem \ref{thm1n}.

Finally, let $f_1$ and $f_2$ be two 1-periodic functions, and let $a_1,a_2>0$. We consider the following system with fast switching rates
\begin{equation}\label{wc}
\begin{cases}
\dot u^\ep_1=-\frac{a_1}{\ep}(u^\ep_1-u^\ep_2)+f_1(\frac{t}{\ep}),& \quad t>0,
\\ \dot u^\ep_2=-\frac{a_2}{\ep}(u^\ep_2-u^\ep_1)+f_2(\frac{t}{\ep}),& \quad t>0,
\\ u^\ep_i(0)=c_i\in\R,&\quad i=1,2.
\end{cases}
\end{equation}
\begin{theorem}\label{thmwc}
For each $c\in\R^2$, let $u^\ep(t;c)$ be the unique solution of \eqref{wc}. Then for all $t>0$, we have
\[|u^\ep_i(t;c)-m_i(t)|\leq \bigg(\frac{2a_2(a_1+a_2)+a_i}{(a_1+a_2)^2}\|f_1\|_\infty+\frac{2a_1(a_1+a_2)+a_i}{(a_1+a_2)^2}\|f_2\|_\infty\bigg)\ep,\]
where
\[m_i(t):=\frac{c_1a_2+c_2a_1}{a_1+a_2}+\frac{a_2\bar f_1+a_1\bar f_2}{a_1+a_2}t+\frac{a_i(c_i-c_j)}{a_1+a_2}e^{-\frac{(a_1+a_2)t}{\ep}}, \quad i,j\in\{1,2\},\ i\neq j,\]
and $\bar f_i=\int_0^1 f_i(t)\, dt$ for $i\in\{1,2\}$.
\end{theorem}

The approximation $(m_1(t),m_2(t))$ is consistent with the result in \cite[Theorem 1.3]{MT}, where the authors established a convergence rate $O(\ep^{1/3})$ for weakly coupled Hamilton--Jacobi systems with fast switching rates. Observe also that
\[
m_1(t)-m_2(t)
=
(c_1-c_2)
e^{-\frac{(a_1+a_2)t}{\ep}}.
\]
Consequently, Theorem \ref{thmwc} implies
\[
|u^\ep_1(t;c)-u^\ep_2(t;c)|
\leq
C\ep
+
|c_1-c_2|
e^{-\frac{(a_1+a_2)t}{\ep}}.
\]
In particular, for every fixed $t>0$,
\[
u^\ep_1(t;c)-u^\ep_2(t;c)\to0
\qquad\text{as }\ep\to0.
\]
Thus the two components collapse onto the same effective trajectory
away from the initial layer.
Such systems admit a random switching interpretation, as discussed in \cite{DSZ}. When $\ep \to 0$, the switching rate becomes very fast and processes have to jump randomly very quickly between the two states according to probabilities determined by $a_1$ and $a_2$. This leads to a weighted average of the data. Our result shows that, when the Hamilton--Jacobi system reduces to an ODE system, the convergence rate is indeed $O(\ep)$. It remains unclear whether the same rate can be achieved in the general setting. Further discussion on this point can be found in Section \ref{4.2}.

\subsection{Application to transport equations}

We denote by Lip$(X,Y)$ (resp. Lip$_{{\rm loc}}(X,Y)$) the space of globally (resp. locally) Lipschitz continuous functions from $X$ to $Y$. When $Y = \mathbb{R}$, we simply write $\operatorname{Lip}(X)$ (resp. $ \operatorname{Lip}_{\mathrm{loc}}(X)$). For $c\in\R^n$, we use the norm $\|c\|=\max_{1\leq i\leq n}|c_i|$. For $\varphi\in{\rm Lip}(\R^n)$, we write $\|D\varphi\|_\infty$ for the smallest constant $C>0$ such that $|\varphi(x)-\varphi(y)|\leq C\|x-y\|$ for all $x,y\in\R^n$.

Let $f\in{\rm Lip}(\R^n\times\R,\R^n)$. For every $\ep>0$, we consider the weak solution $V^\ep\in {\rm Lip}_{\rm loc}(\R^n\times[0,\infty))$ of the linear transport equation
\begin{equation}\label{transeq}
\begin{cases}
V^\ep_t+f(\frac{x}{\ep},\frac{t}{\ep})\cdot DV^\ep=0,\quad &\text{for }x\in\R^n,\ t>0,
\\ V^\ep(x,0)=\varphi(x)\in {\rm Lip}_{\rm loc}(\R^n),\quad &\text{for }x\in\R^n.
\end{cases}
\end{equation}

In \cite{DG}, the author raised the following conjecture:

\medskip

\noindent {\bf Conjecture.} There exists $\overline{V}\in {\rm Lip}_{\rm loc}(\R^n\times[0,\infty))$ such that $\lim_{\ep\to 0}V^\ep=\overline V$ in the weak topology of $L^p_{\rm loc}(\R^n\times[0,\infty))$, and there is a probability measure $\mu$ depending only on $f$ such that, for every $\varphi$, it follows that $\overline{V}(x,t)=\int_{\R^n}\varphi(x-\xi t)\, d\mu(\xi)$.

\medskip

The following result provides the strong convergence of $V^\ep$, with an explicit $O(\ep)$ convergence rate. While quantitative homogenization results for linear transport equations with shear flow characteristics were obtained in \cite{T97}, our setting is different in scope. In \cite{T97}, the vector field $f$ is defined on the two-dimensional torus, possesses an invariant measure, and is nonvanishing. By \cite{Ko}, the flow of \eqref{ven} is conjugate (via a change of coordinates) to a shear flow. In contrast, we work in arbitrary dimension $n\geq 2$ and do not assume the vector field to be of shear flow type. Irrational shear flows arise as a particular case of our framework, as shown in Proposition \ref{prop:shear}.
\begin{theorem}\label{thmt}
Let $\varphi\in{\rm Lip}(\R^n)$, and let $V^\ep$ be the weak solution of \eqref{transeq}. Then
\begin{itemize}
\item[(1)] Assume $n>1$, and conditions {\rm (f4)} and {\rm (f5)} hold. Then there exists $\bar \xi \in\R^n$ such that for all $t>0$ and $x\in\R^n$, 
\[|V^\ep(x,t)-\varphi(x-\bar \xi t)|\leq \|D\varphi\|_\infty \Big(1+C_0+2\max_{1\leq i\leq n}\|f_i\|_\infty\Big)\ep.\]
\item[(2)] Assume $n=1$.
\begin{itemize}
\item [(2-1)] If condition {\rm (f1)} holds, then there exists $\bar \xi \in\R$ such that for all $t>0$ and $x\in\R$, 
\[|V^\ep(x,t)-\varphi(x-\bar \xi t)|\leq \|D\varphi\|_\infty (1+2\|f\|_\infty)\ep.\]
Moreover, if $f(r,\tau)$ is independent of $\tau$, the estimate improves to
\[|V^\ep(x,t)-\varphi(x-\bar \xi t)|\leq \|D\varphi\|_\infty \ep.\]
\item [(2-2)] If $f(r,\tau)$ is independent of $\tau$ and condition {\rm (f3)} holds, then, with $\bar\xi=M_f$ as in Theorem \ref{qp},
there exists $C>0$ depending on $N$, $m$, $F$,
and the Diophantine data of $\xi_0$ such that for all $t>0$
and $x\in\R$,
\[
|V^\ep(x,t)-\varphi(x-\bar \xi t)|
\leq \|D\varphi\|_\infty C\ep.
\]
\end{itemize}
\end{itemize}
\end{theorem}

\subsection*{Brief review of the literature}

The study of the homogenization of ODEs dates back to the classical work \cite{P1}. This topic is closely related to the homogenization of the linear transport equation via the method of characteristics. For developments in both directions, we refer the reader to \cite{Pe96,Pe03,IM, KS, MMT, P, T97}, and the references therein. In particular, \cite{KS} discusses the connection with De Giorgi's conjecture. See also \cite{E, HX} for weak convergence results related to the homogenization of linear transport equations. For recent work relating rotation sets, fine asymptotics of periodic
ODE flows, and homogenization of transport equations in higher
dimensions, we refer to \cite{BH21,BH22,BH23}.

Recently developed techniques for the quantitative homogenization of Hamilton--Jacobi equations have played a significant role in this paper. The study of the homogenization of Hamilton--Jacobi equations began with the seminal work \cite{LPV}. A convergence rate $O(\ep^{1/3})$ for non-convex Hamilton--Jacobi equations was first obtained in \cite{CDI}. In the convex setting, where the optimal control formula for viscosity solutions applies, the first quantitative estimate was given in \cite{MTY}. Later, in \cite{TY}, by employing a refined “curve surgery” argument, the authors established the optimal convergence rate $O(\ep)$ for convex Hamilton--Jacobi equations. This method is robust and has inspired a number of follow-up works, see, for example, \cite{HJ, HJMT, HTZ, MN1, MN2, NT, T}, and the references therein.

Let $u^\ep$ be the unknown function. The case where the Hamilton--Jacobi equation depends periodically on $u^\ep/\ep$ is more closely related to the topic of this paper. Relevant results in this direction can be found in \cite{AP, B, IM08}. In particular, by adapting the curve surgery technique, \cite{MNT} achieved the optimal convergence rate $O(\ep)$ for this problem. This paper is strongly inspired by the approach taken in \cite{MNT}.

\subsection*{Organization of the paper}

Sections \ref{2.1}, \ref{2.2}, and \ref{2.3} are devoted to the proofs of Theorems \ref{thm1}, \ref{thm2}, and \ref{qp}, respectively. In Section \ref{sec:highdim}, we prove Theorems \ref{thmequ} and \ref{thm1n}. Finally, as an application of the quantitative homogenization results for ODEs established in this paper, Section \ref{4.1} is dedicated to the proof of Theorem \ref{thmt}. Further applications of the results obtained in this paper are discussed in Section {4.2}.

\section{One dimensional case}

\subsection{Proof of Theorem \ref{thm1}}\label{2.1}
 
In this section, we assume that condition (f1) holds. We begin by establishing an almost-subadditive estimate for $(\tau,c)\mapsto v(\tau;c)$ in Lemma \ref{lem1}. After adding a fixed constant, the corresponding quantity becomes subadditive. Using Fekete's lemma, we then prove the existence of the limit $\lim_{\ep\to 0}u^\ep(t;c)$ in Lemma \ref{lem2}. Next, we establish the existence of the effective constant $\bar f\in\R$ in Lemma \ref{lem:barf}. Finally, we prove Lemma \ref{lem3}, which immediately yields Theorem \ref{thm1}.

\begin{lemma}\label{lem1}
For all $\sigma,l,t>0$, and all $c\in\R$, we have
\[v((\sigma+l)t;(\sigma+l)c)\leq v(\sigma t;\sigma c)+v(lt;lc)+2(\|f\|_\infty+1).\]
\end{lemma}
\begin{proof}
We first choose $n_1\in\mathbb Z$ such that
\[(\sigma+l)c<lc+n_1\leq (\sigma+l)c+1.\]
By the comparison principle, it follows that
\[v(s;(\sigma+l)c)<v(s;lc+n_1)=v(s;lc)+n_1.\]
This implies
\begin{equation}\label{1}
v(lt;(\sigma+l)c)-(\sigma+l)c\leq v(lt;lc)+n_1-(lc+n_1)+1.
\end{equation}
Next, we choose $n_2\in\mathbb Z$ such that
\[v(\lfloor lt\rfloor +1;(\sigma+l)c)<\sigma c+n_2\leq v(\lfloor lt\rfloor+1;(\sigma+l)c)+1.\]
Define
\[w(s+\lfloor lt \rfloor+1):=v(s;\sigma c)+n_2.\]
Then for $s\in[0,\sigma t]$, we have
\[\dot w(s+\lfloor lt \rfloor+1)=\dot v(s;\sigma c)=f(v(s;\sigma c),s)=f(w(s+\lfloor lt \rfloor+1),s+\lfloor lt \rfloor+1).\]
Note that $\sigma t+\lfloor lt \rfloor+1>(\sigma+l)t$. By the comparison principle, we obtain
\[v((\sigma +l)t;(\sigma+l)c)<w((\sigma+l)t),\]
 which implies
\begin{equation}\label{2}
v((\sigma +l)t;(\sigma+l)c)-v(\lfloor lt\rfloor +1;(\sigma+l)c)<w((\sigma+l) t)-(\sigma c+n_2)+1.
\end{equation}
Also, we have the bounds
\begin{equation}\label{3}
v(\lfloor lt\rfloor +1;(\sigma+l)c)-v(lt;(\sigma+l)c)\leq \|f\|_\infty,
\end{equation}
and
\begin{equation}\label{4}
-\|f\|_\infty\leq w(\lfloor lt\rfloor +1+\sigma t)-w((\sigma+l)t).
\end{equation}
Summing inequalities \eqref{1}-\eqref{4}, we get
\begin{align*}
&v((\sigma +l)t;(\sigma+l)c)-(\sigma+l)c
\\ &\leq w(\lfloor lt\rfloor +1+\sigma t)-(\sigma c+n_2)+v(lt;lc)-lc+2(\|f\|_\infty+1)
\\ &=v(\sigma t;\sigma c)-\sigma c+v(lt;lc)-lc+2(\|f\|_\infty+1),
\end{align*}
which completes the proof.
\end{proof}

\begin{lemma}\label{lem2}
For each $t>0$ and $c\in\R$, the following limit exists
\[\overline{u}(t;c):=\lim_{\ep\to 0}u^\ep(t;c)=\lim_{\ep\to 0}\ep v\Big(\frac{t}{\ep};\frac{c}{\ep}\Big)=\inf_{k>0}\frac{v(kt;kc)+2(\|f\|_\infty+1)}{k}.\]
\end{lemma}
\begin{proof}
The proof is similar to that of Fekete's lemma. For fixed $t>0$ and $c\in\R$, define  
\[\phi(k):=v(kt;kc)+C,\] 
where $k>0$ and 
$C:=2(\|f\|_\infty+1)$.
Then, by Lemma \ref{lem1}, for all $\sigma,l>0$, we have
\begin{equation}\label{phik}
\begin{aligned}
\phi(\sigma+l)&=\, 
v\big((\sigma+l)t;(\sigma+l)c\big)+C\\
&\le\, 
v(\sigma t;\sigma c)+v(lt;lc)+2C
=\phi(\sigma)+\phi(l). 
\end{aligned}
\end{equation}
Let
\[\phi_{\inf}:=\inf_{k>0}\frac{\phi(k)}{k}.\]
By definition, for each $a>\phi_{\inf}$, there exists $l_0>0$ such that
\[\frac{\phi(l_0)}{l_0}\leq a.\]
Now, for any $k>0$, we choose $n\in\mathbb N\cup\{0\}$ such that $nl_0<k\leq (n+1)l_0$. Since $k-nl_0\in(0,l_0]$, there exists a constant $M>0$ independent of $n$ such that $\phi(k-nl_0)\leq M$. Using subadditivity \eqref{phik}, we obtain
\[\phi(k)\leq \phi(nl_0)+\phi(k-nl_0)\leq n\phi(l_0)+M,\]
and therefore, 
\[\frac{\phi(k)}{k}< 
\frac{\phi(l_0)}{l_0}+\frac{M}{nl_0}\leq a+\frac{M}{nl_0}.\]
Letting $n\to\infty$, and then $a\to \phi_{\inf}$, we conclude that
\[
\lim_{k\to\infty}\frac{\phi(k)}{k}=\phi_{\inf}=\inf_{k>0}\frac{\phi(k)}{k}.
\]
Recalling the definition of $\phi$, we obtain the existence of the limit $\lim_{\ep\to 0}\ep v(\frac{t}{\ep};\frac{c}{\ep})$, which completes the proof.
\end{proof}

\begin{lemma}\label{lem:barf}
For all $t>0$ and $c\in\R$, we have
\[\overline{u}(t;c)=\bar f t+c,\]
where
\begin{equation}\label{barf=}
\bar f:=\overline{u}(1;0)=\lim_{k\to\infty}\frac{v(k;0)}{k}=\inf_{k>0}\frac{v(k;0)+2(\|f\|_\infty+1)}{k}.
\end{equation}
Moreover, for all $c\in\R$, we have
\begin{equation}\label{rotat}
\lim_{k\to\infty}\frac{v(k;c)}{k}=\bar f.
\end{equation}
\end{lemma}
\begin{proof}
By periodicity of $f$, we have
\[v(\tau;c+1)=v(\tau;c)+1.\]
For each $c\in\R$, we take $\ep_n\to 0$ such that $c/\ep_n\in\mathbb Z$. Then, 
\[\ep_n v\Big(\frac{t}{\ep_n};\frac{c}{\ep_n}\Big)=\frac{\ep_n}{t} v\Big(\frac{t}{\ep_n};0\Big)\cdot t+c\to \overline{u}(1;0)t+c,\quad \text{as }\ep_n\to 0.\]
By Lemma \ref{lem2}, this implies \eqref{barf=}.

For \eqref{rotat}, we fix $c\in\R$, and take $n\in\mathbb Z$ such that $n\leq c<n+1$. Then
\[v(k;0)+n=v(k;n)\leq v(k;c)\leq v(k;n+1)=v(k;0)+n+1.\]
Dividing by $k$, and sending $k\to\infty$, we get \eqref{rotat}.
\end{proof}

\begin{lemma}\label{lem3}
For all $\tau>0$ and $c\in\R$, we have
\[|v(\tau;c)-(c+\bar f \tau)|\leq 1+2\|f\|_\infty.\]
Moreover, if $f(r,\tau)$ is independent of $\tau$, we have
\[|v(\tau;c)-(c+\bar f \tau)|\leq 1.\]
\end{lemma}
\begin{proof}
For $\tau>0$, we define
\[M_\tau:=\sup_{c\in\R}(v(\tau;c)-c),\quad m_\tau:=\inf_{c\in\R}(v(\tau;c)-c).\]
Using the periodicity of $f$, we have
\[v(\tau;c+1)-(c+1)=v(\tau;c)+1-(c+1)=v(\tau;c)-c.\]
So it suffices to take the supremum and infimum over any interval of length 1, that is, 
\[M_\tau=\sup_{c\in I}(v(\tau;c)-c),\quad m_\tau=\inf_{c\in I}(v(\tau;c)-c),\]
where the length of $I$ is 1. By the basic theory of ODEs, the map $c\mapsto v(\tau;c)-c$ is continuous for all $\tau>0$, so it takes its maximum and minimum. Let $c_1,c_2\in\R$ with $c_2\leq c_1<c_2+1$, $v(\tau;c_2)-c_2=M_\tau$ and $v(\tau;c_1)-c_1=m_\tau$. We have
\begin{equation}\label{Mm}
v(\tau;c_2)-c_2-1\leq v(\tau;c_1)-c_1,
\end{equation}
which implies $M_\tau-m_\tau\leq 1$.

Now observe that, since $f(r,\tau)$ is 1-periodic in $\tau$, for any $\tau,i\in\mathbb N$, we have $v(\tau;v(i\tau;c))=v((i+1)\tau;c)$. Then by using \eqref{rotat} and $M_\tau-m_\tau \leq 1$, for $m\in\mathbb N$, we get
\begin{equation}\label{diff}
\begin{aligned}
|v(\tau;c)-(c+\bar f\tau)|&=\bigg|v(\tau;c)-c-\lim_{m\to \infty}\frac{v(m\tau;c)-c}{m}\bigg|
\\ &=\bigg|v(\tau;c)-c-\lim_{m\to \infty}\frac{1}{m}\sum_{i=0}^{m-1}\Big(v(\tau;v(i\tau;c))-v(i\tau;c)\Big)\bigg|\leq 1,
\end{aligned}
\end{equation}
where the average term satisfies
\[m_\tau\leq \lim_{m\to \infty}\frac{1}{m}\sum_{i=0}^{m-1}\Big(v(\tau;v(i\tau;c))-v(i\tau;c)\Big)\leq M_\tau.\]
If $f(r,\tau)$ is independent of $\tau$, then the inequality \eqref{diff} above holds for all $\tau>0$, and the proof is complete in this case. Otherwise, we note from \eqref{e1} and $\bar f=\lim\limits_{\ep\to 0}u^\ep(1;0)$ that $|\bar f|\leq \|f\|_\infty$. Then, for general $\tau>0$, we write
\begin{align*}
|v(\tau;c)-(c+\bar f\tau)|&=|v(\tau;c)-v(\lfloor \tau \rfloor;c)|+|v(\lfloor \tau \rfloor;c)-(c+\bar f\lfloor \tau \rfloor)|+|\bar f(\lfloor \tau \rfloor-\tau)|
\\ &\leq 1+2\|f\|_\infty.
\end{align*}
This completes the proof.
\end{proof}

\medskip

\noindent {\it Proof of Theorem \ref{thm1}.}
By \eqref{barf=} and \eqref{rotat}, we obtain the identity \eqref{infv}. Applying Lemma \ref{lem3}, for all $t,\ep>0$ and $c\in\R$, we have
\[\bigg|v\Big(\frac{t}{\ep};\frac{c}{\ep}\Big)-\Big(\frac{c}{\ep}+\frac{\bar f t}{\ep}\Big)\bigg|\leq 1+2\|f\|_\infty,\]
which implies
\[|u^\ep(t;c)-(c+\bar f t)|=\bigg|\ep v\Big(\frac{t}{\ep};\frac{c}{\ep}\Big)-(c+\bar f t)\bigg|\leq (1+2\|f\|_\infty)\ep.\]
If $f(r, \tau)$ is independent of $\tau$, the same argument with Lemma \ref{lem3} yields the sharper bound $\varepsilon$. This completes the proof.
\qed

\subsection{Proof of Theorem \ref{thm2}}\label{2.2}

In this section, we assume that condition (f2) holds.
\begin{lemma}\label{mc}
There exists $d_0>0$ such that, for any $t_1,t_2>0$ and
$u_1,u_2\in\R$ satisfying
\[
0<|u_1-u_2|+|t_1-t_2|\le d_0,
\]
we have
\[|\bar f(u_1,t_1)-\bar f(u_2,t_2)|\leq \frac{2}{\Big\lfloor \frac{1}{\kappa}W\Big(\frac{1}{|u_1-u_2|+|t_1-t_2|}\Big)\Big\rfloor}.\]
\end{lemma}
\begin{proof}
The argument is inspired by \cite[Theorem 1.9 (3)]{NWY}. For $i=1,2$, we denote by $v(\tau;c,u_i,t_i)$ the solution of
\begin{equation*}
\begin{cases}
\dot v=f(v,\tau,u_i,t_i),\quad \tau>0,
\\ v(0)=c\in\R.
\end{cases}
\end{equation*}
For each $n\in\mathbb N$, we define
\[M_n:=\sup_{c\in\R}(v(n;c,u_1,t_1)-c),\quad m_n:=\inf_{c\in\R}(v(n;c,u_1,t_1)-c).\]
Similar to \eqref{Mm}, one can obtain $M_n-m_n\leq 1$. By definition, we have
\begin{equation}\label{ab}
m_n\leq v(n;c,u_1,t_1)-c\leq M_n,\quad \forall c\in\R.
\end{equation}
Note that for all $\tau\in[0,n]$, 
\begin{align*}
&|v(\tau;c,u_1,t_1)-v(\tau;c,u_2,t_2)|
\\&=\bigg|\int_0^\tau \Big(f(v(s;c,u_1,t_1),s,u_1,t_1)-f(v(s;c,u_2,t_2),s,u_2,t_2)\Big)\, ds\bigg|
\\ &\leq \int_0^\tau\kappa\Big(|v(s;c,u_1,t_1)-v(s;c,u_2,t_2)|+|u_1-u_2|+|t_1-t_2|\Big)\, ds
\\ &\leq \int_0^\tau\kappa |v(s;c,u_1,t_1)-v(s;c,u_2,t_2)|\, ds+\kappa (|u_1-u_2|+|t_1-t_2|)n.
\end{align*}
By the Gronwall inequality, we get
\[|v(n;c,u_1,t_1)-v(n;c,u_2,t_2)|\leq (|u_1-u_2|+|t_1-t_2|)\kappa ne^{\kappa n},\]
which implies
\begin{equation}\label{ab2}
m_n-(|u_1-u_2|+|t_1-t_2|)\kappa ne^{\kappa n}\leq v(n;c,u_2,t_2)-c\leq M_n+(|u_1-u_2|+|t_1-t_2|)\kappa ne^{\kappa n},
\end{equation}
for all $c\in\R$. Note that for all $q\in\mathbb N$ and for $i=1,2$, we have
\[v(qn;0,u_i,t_i)=\sum_{j=0}^{q-1}\Big(v((j+1)n;0,u_i,t_i)-v(jn;0,u_i,t_i)\Big),\]
where by the periodicity of $f$, we have
\[v((j+1)n;0,u_i,t_i)=v(n;v(jn;0,u_i,t_i),u_i,t_i).\]
Then by \eqref{ab} and \eqref{ab2}, we get
\begin{align*}
&v((j+1)n;0,u_1,t_1)-v(jn;0,u_1,t_1)-\Big(v((j+1)n;0,u_2,t_2)-v(jn;0,u_2,t_2)\Big)
\\ &=v(n;v(jn;0,u_1,t_1),u_1,t_1)-v(jn;0,u_1,t_1)
\\ &\hspace{6.45cm}-\Big(v(n;v(jn;0,u_2,t_2),u_2,t_2)-v(jn;0,u_2,t_2)\Big)
\\ &\leq M_n-m_n+(|u_1-u_2|+|t_1-t_2|)\kappa ne^{\kappa n}\leq 1+(|u_1-u_2|+|t_1-t_2|)\kappa ne^{\kappa n}.
\end{align*}
Summing over $j = 0, \dots, q-1$ and dividing by $qn$, we obtain
\begin{align*}
\frac{1}{qn}\Big(v(qn;0,u_1,t_1)-v(qn;0,u_2,t_2)\Big)&\leq \frac{q+(|u_1-u_2|+|t_1-t_2|)q\kappa ne^{\kappa n}}{qn}
\\ &=\frac{1+(|u_1-u_2|+|t_1-t_2|)\kappa ne^{\kappa n}}{n}.
\end{align*}
Sending $q\to \infty$, we deduce
\[\bar f(u_1,t_1)-\bar f(u_2,t_2)\leq \frac{1+(|u_1-u_2|+|t_1-t_2|)\kappa ne^{\kappa n}}{n}.\]

Choose $d_0>0$ sufficiently small so that \[ \frac{1}{\kappa}W\Big(\frac{1}{d_0}\Big)\geq 1. \]
We take
\[n=\Big\lfloor \frac{1}{\kappa}W\Big(\frac{1}{|u_1-u_2|+|t_1-t_2|}\Big)\Big\rfloor\leq \frac{1}{\kappa}W\Big(\frac{1}{|u_1-u_2|+|t_1-t_2|}\Big).\]
Then
\begin{align*}
&(|u_1-u_2|+|t_1-t_2|)\kappa ne^{\kappa n}
\\ &\leq (|u_1-u_2|+|t_1-t_2|)W\Big(\frac{1}{|u_1-u_2|+|t_1-t_2|}\Big)e^{W(\frac{1}{|u_1-u_2|+|t_1-t_2|})}=1,
\end{align*}
and
\[\bar f(u_1,t_1)-\bar f(u_2,t_2)\leq\frac{2}{\Big\lfloor \frac{1}{\kappa}W\Big(\frac{1}{|u_1-u_2|+|t_1-t_2|}\Big)\Big\rfloor}.\]
The same bound holds in the reverse direction by symmetry, completing the proof.
\end{proof}

Since $\bar f$ is continuous, the Cauchy--P\'eano theorem ensures the existence of a solution $\overline{u}(t;c)$ of \eqref{ebar}. Moreover, since $\bar f(u,t)$ is non-increasing in $u$, the uniqueness of the solution follows by the standard comparison principle.

\begin{lemma}\label{Ct}
For all $t>0$ and $c\in\R$, we have
\[\big|u^\ep(t;c)-\overline{u}(t;c)\big|\leq 2\|f\|_\infty t.\]
\end{lemma}
\begin{proof}
Define $u_1(t):=c+\|f\|_\infty t$, then
\begin{equation*}
\begin{cases}
\dot u_1=\|f\|_\infty \geq f(\frac{u_1}{\ep},\frac{t}{\ep},u_1,t),\quad t>0,\\u_1(0)=c.
\end{cases}
\end{equation*}
By the comparison principle, we obtain
\[u^\ep(t;c)\leq c+\|f\|_\infty t.\]
Similarly, 
\[u^\ep(t;c)\geq c-\|f\|_\infty t.\]
On the other hand, we know that for all $u_0\in\R$ and $t_0>0$,
\[u^\ep(1;0,u_0,t_0)\to \bar f(u_0,t_0)\]
as $\ep\to 0$, where $u^\ep(t;c,u_0,t_0)$ is the solution of \eqref{ute}. As before, we have \[-\|f\|_\infty \leq u^\ep(1;0,u_0,t_0)\leq \|f\|_\infty.\] Then we have
\begin{equation}\label{barf}
\|\bar f\| _\infty\leq \|f\|_\infty.
\end{equation}
We conclude
\[|u^\ep(t;c)-\overline{u}(t;c)|\leq |u^\ep(t;c)-c|+|\overline{u}(t;c)-c|\leq 2\|f\|_\infty t.\]
This completes the proof.
\end{proof}

\medskip

\noindent {\it Proof of Theorem \ref{thm2}.} 
We prove the upper bound of $u^\ep(t;c) - \overline{u}(t;c)$. The lower bound follows by a symmetric argument.

Assume $u^\ep(t;c)-\overline{u}(t;c)>0$. Then by continuity, there is $\sigma\in[0,t)$ such that $u^\ep(\sigma;c)=\overline{u}(\sigma;c)$ and $u^\ep(s;c)>\overline{u}(s;c)$ for all $s\in(\sigma,t]$. We choose $\delta>0$ sufficiently small so that \[ (1+\|f\|_\infty)\delta\leq d_0, \] where $d_0$ is given by Lemma \ref{mc}. We take a partition
\[\sigma\leq \sigma+\delta\leq \dots\leq \sigma+N\delta\leq t\leq \sigma+(N+1)\delta.\]
If $t\leq \sigma+\delta$, then $N=0$. For each $j=0,\dots,N$, we define $t_j=\sigma+j\delta$ and $t_{N+1}=t$. Then we decompose the difference
\[u^\ep(t;c)-u^\ep(\sigma;c)=\sum_{j=0}^{N}\Big(u^\ep(t_{j+1};c)-u^\ep(t_j;c)\Big),\]
and
\[\overline{u}(t;c)-\overline{u}(\sigma;c)=\sum_{j=0}^{N}\Big(\overline{u}(t_{j+1};c)-\overline{u}(t_j;c)\Big).\]
For each $j=0,\dots,N$, define $w_j^\ep(t)$ as the solution to
\begin{equation*}
\begin{cases}
\dot w^\ep_j=f(\frac{w^\ep_j}{\ep},\frac{t}{\ep},\overline{u}(t_j;c),t_j)+\kappa (1+\|f\|_\infty)\delta,\quad t>0,\\ w^\ep_j(t_j)=u^\ep(t_j;c).
\end{cases}
\end{equation*}
For $t \in [t_j, t_{j+1}]$, we estimate
\[\dot u^\ep=f\Big(\frac{u^\ep}{\ep},\frac{t}{\ep},u^\ep(t;c),t\Big)\leq f\Big(\frac{u^\ep}{\ep},\frac{t}{\ep},\overline{u}(t;c),t\Big)\leq f\Big(\frac{u^\ep}{\ep},\frac{t}{\ep},\overline{u}(t_j;c),t_j\Big)+\kappa (1+\|f\|_\infty)\delta,\]
where we used the monotonicity of $f$ in $u$ and \eqref{barf}. Thus, by the comparison principle, 
\[u^\ep(t_{j+1};c)-u^\ep(t_j;c)\leq w^\ep_j(t_{j+1})-w^\ep_j(t_j).\]
To control $w_j^\ep$, we shift its initial value into $[0,\ep]$ by taking $n_j \in \mathbb{Z}$ such that $0\leq w^\ep_j(0)+n_j\ep\leq \ep$. Define $\tilde w^\ep_j(t)=w^\ep_j(t)+n_j\ep$, then
\[\dot{\tilde w}^\ep_j=f\Big(\frac{\tilde w^\ep_j}{\ep},\frac{t}{\ep},\overline{u}(t_j;c),t_j\Big)+\kappa(1+\|f\|_\infty) \delta.\]
By the comparison principle, 
\[w^\ep_j(t;0)\leq \tilde w^\ep_j(t)\leq w^\ep_j(t;\ep)=w^\ep_j(t;0)+\ep,\]
where $w^\ep_j(t;c)$ is the solution of
\begin{equation*}
\begin{cases}
\dot w^\ep_j=f(\frac{w^\ep_j}{\ep},\frac{t}{\ep},\overline{u}(t_j;c),t_j)+\kappa (1+\|f\|_\infty)\delta,\quad t>0,\\ w^\ep_j(0)=c\in\R.
\end{cases}
\end{equation*}
It follows that
\begin{equation}\label{uw}
\begin{aligned}
u^\ep(t_{j+1};c)-u^\ep(t_j;c)&\leq w^\ep_j(t_{j+1})-w^\ep_j(t_j)
\\ &=\tilde w^\ep_j(t_{j+1})-\tilde w^\ep_j(t_j) \leq w^\ep_j(t_{j+1};0)+\ep-w^\ep_j(t_j;0).
\end{aligned}
\end{equation}
Applying Theorem \ref{thm1}, we obtain
\[w^\ep_j(t_{j+1};0)-\bar f_jt_{j+1}\leq C_w\ep,\]
and
\[w^\ep_j(t_j;0)-\bar f_jt_j\geq -C_w\ep,\]
where $\bar f_j=\lim\limits_{\ep\to 0}w^\ep_j(1;0)$ and $C_w:=1+2\|f\|_\infty+2\kappa(1+\|f\|_\infty)\delta$. Then by \eqref{uw}, we get
\[u^\ep(t_{j+1};c)-u^\ep(t_j;c)\leq C_\delta \ep+\bar f_j(t_{j+1}-t_j),\]
where $C_\delta:=3+4\|f\|_\infty+4\kappa(1+\|f\|_\infty)\delta$. For each fixed $j$, define \[ F_j(r,\tau,\eta) := f(r,\tau,\overline{u}(t_j;c),t_j)+\kappa\eta, \] and denote by $\overline{F}_j(\eta)$ the corresponding effective drift. Then \[ \overline{F}_j(0) = \bar f(\overline{u}(t_j;c),t_j), \] whereas \[ \overline{F}_j\big((1+\|f\|_\infty)\delta\big) = \bar f_j. \] By the argument of Lemma \ref{mc}, which uses only the
$\kappa$-Lipschitz dependence on the frozen parameters, we obtain \[ |\bar f_j-\bar f(\overline{u}(t_j;c),t_j)| \leq \frac{2}{ \Big\lfloor \frac{1}{\kappa} W\Big( \frac{1}{(1+\|f\|_\infty)\delta} \Big) \Big\rfloor}. \] Again, by Lemma \ref{mc}, we have
\[\int_{t_j}^{t_{j+1}}\Big|\bar f(\overline{u}(t_j),t_j)-\bar f(\overline{u}(s),s)\Big|\, ds\leq  \frac{2\delta}{\Big\lfloor \frac{1}{\kappa}W\Big(\frac{1}{(1+\|f\|_\infty)\delta}\Big)\Big\rfloor}.\]
If $t\leq \sigma+\delta$, we have $N=0$, then
\[
u^\ep(t;c)-\overline{u}(t;c)\leq C_\delta\ep+\int_\sigma^t|\bar f_j-\bar f(\overline{u}(s),s)|\, ds
\leq C_\delta \ep+\frac{4\delta}{\Big\lfloor \frac{1}{\kappa}W\Big(\frac{1}{(1+\|f\|_\infty)\delta}\Big)\Big\rfloor}.
\]
If $t>\sigma+\delta$, since $N+1\leq 2t/\delta$, we have
\begin{align*}
u^\ep(t;c)-\overline{u}(t;c)
&\leq (N+1)\bigg(C_\delta \ep+\frac{4\delta}{\Big\lfloor \frac{1}{\kappa}W\Big(\frac{1}{(1+\|f\|_\infty)\delta}\Big)\Big\rfloor}\bigg)
\\ &\leq \bigg((3+4\|f\|_\infty)\frac{\ep}{\delta}+4\kappa(1+\|f\|_\infty)\ep+\frac{4}{\Big\lfloor \frac{1}{\kappa}W\Big(\frac{1}{(1+\|f\|_\infty)\delta}\Big)\Big\rfloor}\bigg)2t.
\end{align*}
To optimize the above bound, we take $\delta$ such that
\[\frac{1}{\delta}W\Big(\frac{1}{\delta}\Big)\sim\frac{1}{\ep}.\]
Using the asymptotic expansion $W(x)\sim \log x-\log(\log x)+o(1)$ for $x$ large, we take
\[\delta\sim \ep|\log \ep|.\]
Then one can check that
\[\frac{1}{\delta}W\Big(\frac{1}{\delta}\Big)\sim \frac{1}{\ep\log(1/\ep)}\log\Big(\frac{1}{\ep\log(1/\ep)}\Big)\sim \frac{1}{\ep\log(1/\ep)}\log\Big(\frac{1}{\ep}\Big)\sim \frac{1}{\ep},\]
where we use the fact that
\[\log\Big(\frac{1}{\ep\log(1/\ep)}\Big)=\log\Big(\frac{1}{\ep}\Big)-\log\Big(\log\Big(\frac{1}{\ep}\Big)\Big)\sim \log\Big(\frac{1}{\ep}\Big).\]
Then we get
\[u^\ep(t;c)-\overline{u}(t;c)\leq \frac{Ct}{|\log \ep|},\quad t>\ep|\log \ep|,\]
where $C$ depends only on $\|f\|_\infty$ and $\kappa$. For $t\leq \ep|\log \ep|$, using $\frac{1}{\delta} W(\frac{1}{\delta})\sim \frac{1}{\ep}$, we have
\[\frac{4\delta}{\Big\lfloor \frac{1}{\kappa}W\Big(\frac{1}{(1+\|f\|_\infty)\delta}\Big)\Big\rfloor}= O(\ep),\]
which implies
\[u^\ep(t;c)-\overline{u}(t;c)\leq C\ep,\]
where $C$ depends only on $\|f\|_\infty$ and $\kappa$. Combining with Lemma \ref{Ct}, we get
\[u^\ep(t;c)-\overline{u}(t;c)\leq \min\big\{C\ep,2\|f\|_\infty t\big\}.\]
The proof is now complete.
\qed

\subsection{Proof of Theorem \ref{qp}}\label{2.3}

In this section, we assume that condition (f3) holds. 
Since $F > 0$ on $\mathbb{T}^N$, the function $f(r) = F(\xi_0 r)$ is strictly positive on $\mathbb{R}$. Consequently, the solution $u^\ep(t;c)$ of \eqref{he} satisfies $u^\ep(t;c) > c$ for all $t > 0$. Integrating the ODE yields
\begin{equation}\label{t=}
t=\int_{c}^{u^\ep(t;c)}\frac{du}{f(\frac{u}{\ep})}=\ep\int_{c/\ep}^{u^\ep(t;c)/\ep}\frac{dr}{f(r)},
\end{equation}
where we set $r=u/\ep$. Define $G(\xi) := 1/F(\xi)$. Since $F \in C^m(\mathbb{T}^N)$ and $F > 0$, it follows that $G \in C^m(\mathbb{T}^N)\subset C^{m-1,1}(\T^N)$. Set \[ M_G:=\int_{\T^N}G(\xi)\,d\xi. \] Since $m>\sigma_{\xi_0}$, by \cite[Theorem 1.1(iii)]{TZZ}, there exists $C=C(N,m,\xi_0)>0$ such that \[ \left| \frac{\ep}{u^\ep(t;c)-c} \int_{c/\ep}^{u^\ep(t;c)/\ep} G(\xi_0r)\,dr -M_G \right| \leq C\|G\|_{C^{m-1,1}} \frac{\ep}{u^\ep(t;c)-c}. \] Multiplying by $u^\ep(t;c)-c$ and using \eqref{t=} gives \[ \left| t-M_G\big(u^\ep(t;c)-c\big) \right| \leq C\|G\|_{C^{m-1,1}}\ep. \] Since \[ M_f=M_G^{-1} = \left( \int_{\T^N}\frac{1}{F(\xi)}\,d\xi \right)^{-1}, \] we conclude that \[ \left| u^\ep(t;c)-(c+M_ft) \right| \leq M_f C\|G\|_{C^{m-1,1}}\ep. \] This completes the proof.

\section{Higher dimensional case}\label{sec:highdim}

In this section, we consider the homogenization problem in higher dimensions. We assume that conditions \textnormal{(f4)} and \textnormal{(f5)} are satisfied.

\begin{lemma}\label{lem0n}
Condition {\rm (f5)} implies \eqref{bdtn}.
\end{lemma}
\begin{proof}
Let $c^1, c^2 \in \mathbb{R}^n$ with $c^1 - c^2 \in Y$. For each $i=1,\dots,n$, define
\[\bar c^2_i:=c^2_i-\lfloor c^2_i-c^1_i\rfloor.\]
Then
\[\bar c^2_i-c^1_i\in [0,1),\]
that is, $\bar c^2-c^1\in Y$. Since $c^1_i-c^2_i\in[0,1]$, we have
\[-1\leq \lfloor c^2_i-c^1_i\rfloor \leq 0.\]
By \eqref{C0}, we have
\[v_i(\tau;c^1)-v_i(\tau;\bar c^2)\leq C_0,\]
which implies
\[v_i(\tau;c^1)-v_i(\tau;c^2)\leq -\lfloor c^2_i-c^1_i\rfloor+C_0\leq 1+C_0.\]
\end{proof}

\begin{lemma}\label{lem1n}
For all $\sigma,l,t>0$ and for all $c\in\R^n$, we have
\[v_i((\sigma+l)t;(\sigma+l)c)\leq v_i(\sigma t;\sigma c)+v_i(lt;lc)+2(\|f_i\|_\infty+C_0+1).\]
\end{lemma}
\begin{proof}
We first take $n^1\in\mathbb Z^n$ such that
\[lc+n^1-(\sigma+l)c\in Y.\]
By \eqref{C0}, it follows that
\[v_i(s;(\sigma+l)c)-(v_i(s;lc)+n^1_i)\leq C_0.\]
Thus, 
\begin{equation}\label{1n}
v_i(lt;(\sigma+l)c)-(\sigma+l)c_i\leq v_i(lt;lc)+n^1_i+C_0-(lc_i+n^1_i)+1.
\end{equation}
Next, we take $n^2\in\mathbb Z^n$ such that
\[\sigma c+n^2-v(\lfloor lt\rfloor+1;(\sigma+l)c)\in Y.\]
Define
\[\tilde w(s)=w(s+\lfloor lt\rfloor +1):=v(s;\sigma c)+n^2,\quad s\in[0,\sigma t].\]
One can easily check that for $s\in[0,\sigma t]$, 
\[\tilde w(s)=v(s;\sigma c+n^2),\quad \tilde v(s):=v(s+\lfloor lt \rfloor+1;(\sigma+l)c)=v(s;v(\lfloor lt \rfloor+1;(\sigma+l)c)).\]
By applying \eqref{C0} to $\tilde w(s)$ and $\tilde v(s)$, we get 
\[v_i((\sigma +l)t;(\sigma+l)c)\leq w_i((\sigma+l)t)+C_0,\]
 which implies
\begin{equation}\label{2n}
v_i((\sigma +l)t;(\sigma+l)c)-v_i(\lfloor lt\rfloor +1;(\sigma+l)c)\leq w_i((\sigma+l) t)+C_0-(\sigma c_i+n^2_i)+1.
\end{equation}
Additionally, we have
\begin{equation}\label{3n}
v_i(\lfloor lt\rfloor +1;(\sigma+l)c)-v_i(lt;(\sigma+l)c)\leq \|f_i\|_\infty.
\end{equation}
and
\begin{equation}\label{4n}
-\|f_i\|_\infty\leq w_i(\lfloor lt\rfloor +1+\sigma t)-w_i((\sigma+l)t).
\end{equation}
Adding inequalities \eqref{1n}-\eqref{4n}, we conclude
\begin{align*}
&v_i((\sigma +l)t;(\sigma+l)c)-(\sigma+l)c_i
\\ &\leq w_i(\lfloor lt\rfloor +1+\sigma t)-(\sigma c_i+n^2_i)+v_i(lt;lc)-lc_i+2(\|f_i\|_\infty+C_0+1)
\\ &=v_i(\sigma t;\sigma c)-\sigma c_i+v_i(lt;lc)-lc_i+2(\|f_i\|_\infty+C_0+1),
\end{align*}
which completes the proof.
\end{proof}

\begin{lemma}
For each $i = 1,\dots,n$, the following limit exists
\[\overline{u}_i(t;c):=\lim_{\ep\to 0}u^\ep_i(t;c)=\lim_{\ep\to 0}\ep v_i\Big(\frac{t}{\ep};\frac{c}{\ep}\Big)=\inf_{k>0}\frac{v_i(kt;kc)+2(\|f_i\|_\infty+C_0+1)}{k}.\]
Moreover, for all $c\in\R^n$, we have
\begin{equation}\label{rotatn}
\lim_{k\to\infty}\frac{v_i(k;c)}{k}=\bar f_i,
\end{equation}
where
\begin{equation}\label{barf=n}
\bar f_i:=\overline{u}_i(1;0)=\lim_{k\to\infty}\frac{v_i(k;0)}{k}=\inf_{k>0}\frac{v_i(k;0)+2(\|f_i\|_\infty+C_0+1)}{k}.
\end{equation}
\end{lemma}
\begin{proof}
Similar to Lemma \ref{lem2}, the limit $\lim_{\ep\to 0}u^\ep(t;c)$ exists. For each $c\in\R^n$, we take $n\in\mathbb Z^n$ such that $c-n\in Y$, then
\[v_i(k;0)+n_i=v_i(k;n)\leq v_i(k;c)+C_0,\]
and
\[v_i(k;c)\leq v_i(k;n)+1+C_0=v_i(k;0)+n_i+1+C_0.\]
Dividing both inequalities by $k$ and letting $k \to \infty$, we conclude \eqref{rotatn}.
\end{proof}

\medskip

\noindent {\it Proof of Theorem \ref{thmequ}.}
By Lemma \ref{lem0n}, we know that \eqref{C0} implies \eqref{bdtn}.

For $\tau>0$ and for each $i=1,\dots,n$, we define
\[M^i_\tau:=\sup_{c\in\R^n}(v_i(\tau;c)-c_i),\quad m^i_\tau:=\inf_{c\in\R^n}(v_i(\tau;c)-c_i).\]
Thanks to the periodicity of $f$, we have for all $n \in \mathbb{Z}^n$,
\[v_i(\tau;c+n)-(c_i+n_i)=v_i(\tau;c)+n_i-(c_i+n_i)=v_i(\tau;c)-c_i,\]
which implies that it suffices to consider $c$ in a fundamental domain $I := c^0 + Y$ for some fixed $c^0 \in \mathbb{R}^n$. Thus,
\[M^i_\tau=\sup_{c\in I}(v_i(\tau;c)-c_i),\quad m^i_\tau=\inf_{c\in I}(v_i(\tau;c)-c_i).\]
Since for each $i=1,\dots,n$, the map $c\mapsto v_i(\tau;c)-c_i$ is continuous for all $\tau>0$, the above supremum and infimum can be achieved. By periodicity, we may choose $c^1,c^2\in\R^n$ with
$c^1-c^2\in Y$ such that $v_i(\tau;c^2)-c^2_i=M^i_\tau$ and $v_i(\tau;c^1)-c^1_i=m^i_\tau$. By \eqref{C0}, we have
\[v_i(\tau;c^2)-c^2_i-1\leq v_i(\tau;c^1)-c^1_i+C_0.\]
Hence, $M^i_\tau-m^i_\tau\leq 1+C_0$.

Note that for $\tau,j\in\mathbb N$, we have $v_i(\tau;v(j\tau;c))=v_i((j+1)\tau;c)$. Then by using \eqref{rotatn} and $M^i_\tau-m^i_\tau \leq 1+C_0$, for $m\in\mathbb N$, we get
\begin{equation}\label{1+C0}
\begin{aligned}
|v_i(\tau;c)-(c_i+\bar f_i\tau)|&=\bigg|v_i(\tau;c)-c_i-\lim_{m\to \infty}\frac{v_i(m\tau;c)-c_i}{m}\bigg|
\\ &=\bigg|v_i(\tau;c)-c_i-\lim_{m\to \infty}\frac{1}{m}\sum_{j=0}^{m-1}\Big(v_i(\tau;v(j\tau;c))-v_i(j\tau;c)\Big)\bigg|
\\ &\leq 1+C_0.
\end{aligned}
\end{equation}
If $f_i(r,\tau)$ is independent of $\tau$ for all
$i=1,\ldots,n$, the above inequality holds for all $\tau>0$, and the proof is complete in this case. If not, since $\bar f_i=\lim\limits_{\ep\to 0}u^\ep_i(1;0)$, by \eqref{uen} we have $|\bar f_i|\leq \|f_i\|_\infty$. Then we estimate for general $\tau > 0$, 
\begin{align*}
|v_i(\tau;c)-(c_i+\bar f_i\tau)|&=|v_i(\tau;c)-v_i(\lfloor \tau \rfloor;c)|+|v_i(\lfloor \tau \rfloor;c)-(c_i+\bar f_i\lfloor \tau \rfloor)|+|\bar f_i(\lfloor \tau \rfloor-\tau)|
\\ &\leq 1+C_0+2\|f_i\|_\infty.
\end{align*}
This completes the proof of the uniform bound. By combining with \eqref{rotatn} and \eqref{barf=n}, we obtain \eqref{infvn}.
\qed

\bigskip

\noindent {\it Proof of Theorem \ref{thm1n}.} By Theorem \ref{thmequ}, for all $t,\ep>0$ and $c\in\R^n$, we have
\[\bigg|v_i\Big(\frac{t}{\ep};\frac{c}{\ep}\Big)-\Big(\frac{c_i}{\ep}+\frac{\bar f_i t}{\ep}\Big)\bigg|\leq 1+C_0+2\|f_i\|_\infty.\]
Multiplying both sides by $\varepsilon$ yields
\[|u^\ep_i(t;c)-(c_i+\bar f_i t)|=\bigg|\ep v_i\Big(\frac{t}{\ep};\frac{c}{\ep}\Big)-(c_i+\bar f_i t)\bigg|\leq (1+C_0+2\|f_i\|_\infty)\ep.\]
If $f_i(r,\tau)$ is independent of $\tau$ for all
$i=1,\ldots,n$, the same estimate holds with a smaller constant as in Theorem~\ref{thmequ}. This completes the proof.
\qed

\begin{prop}\label{prop:shear}
Let $n\geq 2$, and let $\xi\in\R^n$ be diophantine, with $C_\xi,\sigma>0$ such that
\[|\xi \cdot k|\geq C_\xi|k|^{-(n+\sigma)}\,\quad \text{for all }k\in\Z^n\backslash\{0\}.\]
Assume $G\in C^m(\R^n)$ is $\Z^n$-periodic, $m>2n+\sigma$ and $G>0$. Let $u^\ep(t;c)$ and $v(\tau;c)$ be the unique solution of
\begin{equation*}
\begin{cases}
\dot u^\ep=\frac{\xi}{G(\frac{u^\ep}{\ep})},\quad t>0,\\ u^\ep(0)=c\in\R^n,
\end{cases}
\end{equation*}
and
\begin{equation*}
\begin{cases}
\dot v=\frac{\xi}{G(v)},\quad \tau>0,\\ v(0)=c\in\R^n,
\end{cases}
\end{equation*}
respectively. Then condition {\rm (f5)} is satisfied for the irrational shear flow $v(\tau;c)$. As a consequence, there is a constant $C>0$ and a vector $\bar f\in\R^n$ depending only on $G$ and $\xi$ such that for all $t>0$ and all $c\in\R^n$,
\[|u^\ep_i(t;c)-(c_i+\bar f_i t)|\leq C\ep.\]
\end{prop}
\begin{proof}
We verify that the rescaled solution $v(\tau;c)$ satisfies condition \eqref{C0}. According to \cite{K}, there is a map $\Phi:v\mapsto w$ defined by
\[w_i=v_i+\frac{\xi_i}{M_G}\theta(v),\]
which maps $v(\tau;c)$ to the linear flow $w(\tau;c)$ solving $\dot w_i=\frac{\xi_i}{M_G}$. Here, $M_G=\int_{\T^n}G(z)\, dz$, and $\theta$ is determined from $D\theta\cdot \xi=G(v)-M_G$. From Fourier's expansion we have
\[\theta(v)=\sum_{k\in\Z^n\backslash\{0\}}\frac{G_k}{2\pi i k\cdot \xi}e^{2\pi i k\cdot v},\]
where $G_k$ is the $k$-th Fourier coefficient of $G-M_G$. Since $G\in C^m(\R^n)$ is $\Z^n$-periodic, we have $|G_k|\leq C(1+|k|)^{-m}$. We recall that $m>2n+\sigma$ and $|\xi \cdot k|\geq C_\xi|k|^{-(n+\sigma)}$. Hence, $\theta$ is bounded by some $C>0$. Since \[ w_i(\tau;c) = c_i+\frac{\xi_i}{M_G}\theta(c) +\frac{\xi_i}{M_G}\tau, \] we obtain \[ v_i(\tau;c) = c_i+\frac{\xi_i}{M_G}\tau +\frac{\xi_i}{M_G} \big( \theta(c)-\theta(v(\tau;c)) \big). \] Hence, if $c^1-c^2\in Y$, then \[ v_i(\tau;c^2)-v_i(\tau;c^1) = c_i^2-c_i^1 + \frac{\xi_i}{M_G} \Big( \theta(c^2)-\theta(c^1) -\theta(v(\tau;c^2)) +\theta(v(\tau;c^1)) \Big)\leq \frac{4C|\xi_i|}{M_G},  \]where $C=\|\theta\|_\infty$.
Therefore \eqref{C0} holds with, for instance,
\[
C_0=
\frac{4C}{M_G}
\max_{1\leq i\leq n}|\xi_i|.
\] We are thus in a position to apply Theorem \ref{thm1n}, which yields the desired estimate.
\end{proof}

\bigskip

\noindent {\it Proof of Theorem \ref{thmwc}}.
We define
\begin{equation*}
\begin{cases}
z^\ep(t)=a_2u^\ep_1(t)+a_1u^\ep_2(t), 
\\ d^\ep(t)=u^\ep_1(t)-u^\ep_2(t).
\end{cases}
\end{equation*}
Then the system becomes
\begin{equation*}
\begin{cases}
\dot z^\ep=a_2f_1(\frac{t}{\ep})+a_1f_2(\frac{t}{\ep}), 
\\ \dot d^\ep=-\frac{a_1+a_2}{\ep}d^\ep+f_1(\frac{t}{\ep})-f_2(\frac{t}{\ep}).
\end{cases}
\end{equation*}
The equation for $d^\varepsilon(t)$ is linear and has the explicit solution
\[d^\ep(t)=(c_1-c_2)e^{-\frac{(a_1+a_2)t}{\ep}}+\int_0^t e^{-\frac{(a_1+a_2)(t - s)}{\ep}}\Big(f_1\Big(\frac{s}{\ep}\Big)-f_2\Big(\frac{s}{\ep}\Big)\Big)\, ds.\]
Using boundedness of $f_1$ and $f_2$, one shows that
\[\bigg|\int_0^t e^{-\frac{(a_1+a_2)(t - s)}{\ep}}\Big(f_1\Big(\frac{s}{\ep}\Big)-f_2\Big(\frac{s}{\ep}\Big)\Big)\, ds\bigg|\leq \frac{\|f_1\|_\infty+\|f_2\|_\infty}{a_1+a_2}\ep.\]
We integrate the equation for $z^\ep(t)$ to get
\[z^\ep(t)=c_1a_2+c_2a_1+\int_0^t \Big(a_2f_1\Big(\frac{s}{\ep}\Big)+a_1f_2\Big(\frac{s}{\ep}\Big)\Big)\, ds.\]
Setting $\tau=s/\ep$, we get
\[\int_0^t f_i\Big(\frac{s}{\ep}\Big)\, ds=\ep\int_0^{t/\ep}f_i(\tau)\, d\tau.\]
Note that for all $t>0$,
\[\bigg|\int_0^t f_i(s)\, ds-\bar f_i t\bigg|=\bigg|\int_{\lfloor t\rfloor}^tf_i(s)\, ds-\bar f_i (t-\lfloor t\rfloor)\bigg|\leq 2\|f_i\|_\infty.\]
Thus,
\begin{align*}
&\bigg|\int_0^t \Big(a_2f_1\Big(\frac{s}{\ep}\Big)+a_1f_2\Big(\frac{s}{\ep}\Big)\Big)\, ds-(a_2\bar f_1+a_1\bar f_2)t\bigg|
\\ &=\bigg|\int_0^{t/\ep} (a_2f_1(\tau)+a_1f_2(\tau))\, d\tau-(a_2\bar f_1+a_1\bar f_2)\frac{t}{\ep}\bigg|\ep\leq 2(a_2\|f_1\|_\infty+a_1\|f_2\|_\infty)\ep.
\end{align*}
We now come back to the original solution
\begin{equation*}
\begin{cases}
u^\ep_1(t)=\frac{z^\ep(t)+a_1d^\ep(t)}{a_1+a_2},
\\ u^\ep_2(t)=\frac{z^\ep(t)-a_2d^\ep(t)}{a_1+a_2}.
\end{cases}
\end{equation*}
Therefore,
\[|u^\ep_i-m_i(t)|\leq \bigg(\frac{2a_2(a_1+a_2)+a_i}{(a_1+a_2)^2}\|f_1\|_\infty+\frac{2a_1(a_1+a_2)+a_i}{(a_1+a_2)^2}\|f_2\|_\infty\bigg)\ep,\quad i=1,2,\]
which completes the proof.
\qed

\section{Applications}\label{Sec4}

\subsection{Proof of Theorem \ref{thmt}}\label{4.1}

Let $u^\ep(t;c)$ be the unique solution of \eqref{uen}. By the method of characteristics, it is well-known that the weak solution of \eqref{transeq} is given by
\[V^\ep(x,t)=\varphi(x_0),\]
where $x_0\in\R^n$ is the unique vector such that $u^\ep(t;x_0)=x$. For each $t>0$, there exist $n\in\mathbb N$ and $l\in[0,\ep)$ such that $t=n\ep+l$. Define $z^\ep(s):=u^\ep(t-s;x_0)$ for $s\in[0,t]$. Then one can check that $z^\ep(s)$ satisfies
\begin{equation*}
\begin{cases}
\dot z^\ep(s)=-f(\frac{z^\varepsilon(s)}{\varepsilon},\frac{t-s}{\ep})=-f(\frac{z^\varepsilon(s)}{\varepsilon},\frac{l-s}{\ep}),
\quad s\in[0,t],
\\ z^\varepsilon(0)=x,\quad z^\ep(t)=x_0.
\end{cases}
\end{equation*}
Define $z^{\ep,l}(s):=z^\ep(s+l)$, then
\begin{equation*}
\begin{cases}
\dot z^{\ep,l}(s)=\dot z^\ep(s+l)=-f(\frac{z^\varepsilon(s+l)}{\varepsilon},\frac{l-(s+l)}{\ep})=-f(\frac{z^{\ep,l}(s)}{\varepsilon},-\frac{s}{\ep}),
\quad s\in[-l,t-l],
\\ z^{\ep,l}(-l)=x,\quad z^{\ep,l}(t-l)=x_0.
\end{cases}
\end{equation*}
Note that $\frac{t-l}{\ep}=n\in\N$, so we may apply the estimate \eqref{1+C0} (after rescaling) to obtain that there is a  vector $\bar \xi\in\R^n$ such that
\[|z^{\ep,l}_i(t-l)-(z^{\ep,l}_i(0)-\bar \xi_i(t-l))|\leq (1+C_0)\ep.\]
Moreover, we have $|\bar \xi_i|\leq \|f_i\|_\infty$, since $\|-f_i(r,-\tau)\|_\infty=\|f_i(r,\tau)\|_\infty$. Then, recall that $l\in[0,\ep)$, we have
\begin{align*}
|z^\ep_i(t)-(x_i-\bar \xi_i t)|&\leq |z^{\ep,l}_i(t-l)-(z^{\ep,l}_i(0)-\bar \xi_i(t-l))|+|z^{\ep,l}_i(0)-x_i|+|\bar \xi_i l|
\\ &\leq (1+C_0+2\|f_i\|_\infty)\ep.
\end{align*}
Thus, using the regularity of $\varphi$, we conclude
\begin{align*}
|V^\ep(x,t)-\varphi(x-\bar \xi t)|&=|\varphi(z^\ep(t))-\varphi(x-\bar \xi t)|\leq \|D\varphi\|_\infty \|z^\ep(t)-(x-\bar \xi t)\|
\\ &\leq \|D\varphi\|_\infty \Big(1+C_0+2\max_{1\leq i\leq n}\|f_i\|_\infty\Big)\ep.
\end{align*}
This proves part (1). Part (2-1) follows similarly from
Theorem \ref{thm1}. For part (2-2), let $x_0\in\R$ be such that
$u^\ep(t;x_0)=x$. By Theorem \ref{qp},
\[
|x-(x_0+M_f t)|
=
|u^\ep(t;x_0)-(x_0+M_f t)|
\leq C\ep.
\]
Hence
\[
|x_0-(x-M_f t)|\leq C\ep,
\]
and therefore
\[
|V^\ep(x,t)-\varphi(x-M_f t)|
\leq \|D\varphi\|_\infty C\ep.
\]

\subsection{Further discussions}\label{4.2}

We first explain how our main results apply to PDEs. Consider the Hamilton--Jacobi equation
\begin{equation}\label{HJE}
\begin{cases}
U^\ep_t+H(\frac{x}{\ep},\frac{U^\ep}{\ep},DU^\ep)=0,\quad &\text{for }x\in\R^n,\ t>0,
\\ U^\ep(x,0)=\varphi(x)\in {\rm Lip}(\R^n),\quad &\text{for }x\in\R^n,
\end{cases}
\end{equation}
where $D$ stands for the derivative with respect to $x$, Lip$(\R^n)$ stands for the set of all globally Lipschitz continuous functions on $\R^n$, and the Hamiltonian $H(y,r,p)$ is continuous, periodic in the first and second variable, and globally Lipschitz continuous in the second variable. The above equation arises in the study of dislocation dynamics in crystals, see \cite[Appendix B]{MNT} and \cite[Section 2]{IMR}. In particular, if $H(y,r,p)$ is superlinear and convex in $p$, \cite{MNT} establishes the optimal convergence rate $O(\ep)$ of $U^\ep$. Let $f$ be Lipschitz continuous and 1-periodic. We denote by $u^\ep(t;c)$ the unique solution of
\begin{equation*}
\begin{cases}
\dot u^\ep=f(\frac{u^\ep}{\ep}),\quad t>0,\\ u^\ep(0)=c\in\R,
\end{cases}
\end{equation*}
Define $U^\ep(x,t):=u^\ep(t;c)$ for all $x\in\R^n$. It is straightforward to verify that $U^\ep(x,t)$ solves the following Hamilton--Jacobi equation
\begin{equation*}
\begin{cases}
U^\ep_t+H(\frac{U^\ep}{\ep},DU^\ep)=0,\quad &\text{for }x\in\R^n,\ t>0,
\\ U^\ep(x,0)=c\in\R,\quad &\text{for }x\in\R^n,
\end{cases}
\end{equation*}
where the Hamiltonian $H:\R\times\R^n\to\R$ satisfies $H(r,0)=-f(r)$. As noted in \cite[Section 2.2]{IMR}, the case with constant initial data $U^\ep(x,0)=c\in\R$ corresponds to a constant initial plastic strain, and the case where the equation is independent of $\frac{x}{\ep}$ corresponds to constant (i.e., spatially homogeneous) obstacles. In this case, \eqref{1ep} shows the sharp estimate \[|U^\ep(x,t)-(c+\bar ft)|\leq \ep,\quad \text{for all }x\in\R^n,\ t>0,\]
where $\bar f=-\ol{H}(0)$, and $\ol{H}$ is the effective Hamiltonian, see \cite{IM08}. Moreover, if we assume (f3), Theorem \ref{qp} may have potential applications to the homogenization of Hamilton--Jacobi equations arising in dislocation dynamics in quasi-crystals, where quasi-periodic structures naturally appear. Similarly, it is clear that the function $U^\ep(x,t)$ also solves the following semilinear heat equation
\begin{equation*}
\begin{cases}
U^\ep_t-\ep^\alpha\Delta U^\ep-f(\frac{U^\ep}{\ep})=0,\quad \text{for }x\in\R^n,\ t>0,
\\ U^\ep(x,0)=c\in\R,
\end{cases}
\end{equation*}
which was studied in \cite{CDN}. Again, the case of constant initial data corresponds to an initially uniform temperature distribution.

Let $a_1,a_2>0$. Consider the following weakly coupled system
\begin{equation}\label{wchx}
\begin{cases}
(U^\ep_1)_t+H_1(\frac{x}{\ep},\frac{t}{\ep},DU^\ep_1)+\frac{a_1}{\ep}(U^\ep_1-U^\ep_2)=0, &\quad\text{for }x\in\R^n,\  t>0,
\\ (U^\ep_2)_t+H_2(\frac{x}{\ep},\frac{t}{\ep},DU^\ep_2)+\frac{a_2}{\ep}(U^\ep_2-U^\ep_1)=0, &\quad\text{for }x\in\R^n,\  t>0,
\\ U^\ep_i(x,0)=\varphi_i(x)\in {\rm Lip}(\R^n), &\quad \text{for }x\in\R^n,\ i=1,2.
\end{cases}
\end{equation}
The above system is weakly coupled in the sense that the $i$-th equation depends on $DU^\ep$ only through $DU^\ep_i$, not on $DU^\ep_j$ for $j\neq i$. This problem was studied in \cite{MT}, and arises naturally in optimal control theory with random switching. By adapting the method in \cite{CDI}, the authors were able to give the convergence rate $O(\ep^{1/3})$ for the viscosity solution $(U^\ep_1,U^\ep_2)$ of \eqref{wchx}. Let $(u^\ep_1(t;c),u^\ep_2(t;c))$ be the unique solution of \eqref{wc}. We can apply Theorem \ref{thmwc} to the following weakly coupled system
\begin{equation}\label{wch}
\begin{cases}
(U^\ep_1)_t+H_1(\frac{t}{\ep},DU^\ep_1)+\frac{a_1}{\ep}(U^\ep_1-U^\ep_2)=0, &\quad\text{for }x\in\R^n,\  t>0,
\\ (U^\ep_2)_t+H_2(\frac{t}{\ep},DU^\ep_2)+\frac{a_2}{\ep}(U^\ep_2-U^\ep_1)=0, &\quad\text{for }x\in\R^n, \ t>0,
\\ U^\ep_i(x,0)=c_i\in\R, &\quad \text{for }x\in\R^n,\ i=1,2,
\end{cases}
\end{equation}
where $H_i(\tau,0)=-f_i(\tau)$ for $i=1,2$. Define $(U^\ep_1(x,t),U^\ep_2(x,t)):=(u^\ep_1(t;c),u^\ep_2(t;c))$ for all $x\in\R^n$. It is straightforward to verify that $(U^\ep_1(x,t),U^\ep_2(x,t))$ solves \eqref{wch}. Then Theorem \ref{thmwc} implies the convergence rate $O(\ep)$ for $(U^\ep_1(x,t),U^\ep_2(x,t))$. For the general case \eqref{wchx}, however, the best known rate is $O(\ep^{1/3})$, and it is an interesting open direction to investigate whether sharper rates can be achieved. In this context, the optimal control formula introduced in \cite{DSZ} may provide a useful tool.

\medskip

In \cite{Po}, the author explained how the problem \eqref{fmul} is connected to a class of free boundary problems arising in PDEs. Specifically, for $\Omega\subset \R^n$ and $T>0$, the author considered the following problem on $\Omega\times (0,T]$
\begin{equation}\label{free}
\begin{cases}
-\Delta U(x,t)=0,\quad &\text{in }\{U>0\},
\\ V_\nu(x,t)=g(\frac{x}{\ep},\frac{t}{\ep})|DU^+(x,t)|\quad &\text{on }\partial\{U>0\},
\\ U(0,t)=\psi(t)\in\Lip([0,\infty))\quad &\text{for }t>0,
\end{cases}
\end{equation}
where $V_\nu$ denotes the normal velocity of the free boundary $\partial\{U>0\}$, $DU^+$ represents the limits of the derivatives from the positive set of $U$, and $g$ is $\Z^{n+1}$-periodic. In the one-dimensional case, we take $\Omega=(0,\infty)$. Assume $g,\psi>0$ and $\psi$ is bounded. The free boundary $\partial\{U>0\}$ can be represented by a single moving point $u^\ep(t)\in(0,\infty)$, with initial condition $u^\ep(0)=c>0$. The normal velocity $V_\nu$ is then given by $\dot u^\ep(t)$. Since $g,\psi>0$, by the second equation in \eqref{free} we know that $\dot u^\ep(t)>0$ and $u^\ep(t)>c>0$. We observe from the first equation in \eqref{free} that $U(x,t)$ is linear in $x$. Since $u^\ep(t)$ is the free boundary, i.e., $U(u^\ep(t),t)=0$, we know that $|Du(u^\ep(t),t)|=\psi(t)/u^\ep(t)$. Then we obtain the following ODE for the motion of the free boundary
\begin{equation*}
\begin{cases}
\dot u^\ep(t)=g(\frac{u^\ep(t)}{\ep},\frac{t}{\ep})\frac{\psi(t)}{u^\ep(t)},
\\ u^\ep(0)=c>0,
\end{cases}
\end{equation*}
which is a special case of \eqref{fmul}. Since $u^\ep(t)>c$, it is clear that condition (f2) is satisfied.

\medskip

We now describe the connection between our model and gradient systems with wiggly energies. In \cite{M}, the author considered the gradient system
\begin{equation}\label{FA}
\dot u^\ep=-\nabla\bigg(F(u^\ep)+\ep A\Big(\frac{u^\ep}{\ep}\Big)\bigg)=-\nabla F(u^\ep)-\nabla A\Big(\frac{u^\ep}{\ep}\Big),
\end{equation}
where $F:\R^n\to\R$ and $A:\T^n\to\R$ are smooth functions. Moreover, the author in \cite{M} assumes that
\begin{equation}\label{F}
F(x)\to \infty\quad \text{as }\|x\|\to \infty.
\end{equation}
This system is motivated by various physical models, such as martensitic phase transformations, gas adsorption, dry friction, etc. As described in \cite{M}, one may imagine a light particle sliding down a rough slope. The macroscopic shape of the slope is modeled by a smooth potential $F(\cdot)$, which drives the overall downward motion. Superimposed on this slope is a rapidly oscillating perturbation $\ep A(\cdot/\ep)$, which encodes the small-scale surface roughness. The particle takes a jerky path downhill, possibly getting stuck along the way. In higher dimensions, one can in general only expect weak convergence of $u^\ep$ to the limiting trajectory $\ol u$. In \cite{M}, the author focuses on deriving the effective equation for $\ol u$ in the case $n=2$. As noted in that work, the effective equation may fail to be Lipschitz continuous. This phenomenon aligns with \eqref{modu00} below.

As an application of Theorem \ref{thm2}, we have the following result.
\begin{prop}
Assume $n=1$. Let $F\in C^2(\R)$ be a convex function satisfying \eqref{F}. Let $A\in C^2(\T)$. For each $u_0\in\R$, we consider the unique solution $u^\ep(t;c,u_0)$ of
\begin{equation*}
\begin{cases}
\dot u^\ep=-F'(u_0)-A'(\frac{u^\ep}{\ep}),\quad t>0,\\ u^\ep(0)=c\in\R.
\end{cases}
\end{equation*}
Then by Theorem \ref{thm1},
\[\lim_{\ep\to 0} u^\ep(t;c,u_0)=c+\bar f(u_0)t.\]
The function $\bar f$ is non-increasing, and admits a modulus of continuity locally. More precisely, for any $u_1,u_2\in U$ with $U\subset\R$ compact, we have
\begin{equation}\label{modu00}
|\bar f(u_1)-\bar f(u_2)|\leq \frac{2}{\Big\lfloor \frac{1}{\max\big\{\max\limits_{x\in U}|F''(x)|,\|A''\|_\infty\big\}}W\Big(\frac{1}{|u_1-u_2|}\Big)\Big\rfloor},
\end{equation}
Moreover, there is a unique solution $\overline{u}(t;c)$ of
\begin{equation*}
\begin{cases}
\dot{\overline{u}}=\bar f(\overline{u}),\quad t>0,\\ \overline{u}(0)=c\in\R.
\end{cases}
\end{equation*}
Let $\mathcal K\subset\R$ be bounded and let $c\in\mathcal K$. Define
\[\mathcal F:=\Big\{x\in\R:\ F(x)\leq \sup_{c\in\mathcal K}F(c)+2\|A\|_\infty\Big\},\]
which is compact by \eqref{F}. Let $u^\ep(t;c)$ be the unique solution of \eqref{FA}. Then there is a constant $C>0$, depending only on $\|F'\|_{W^{1,\infty}(\mathcal F)}$ and $\|A'\|_{W^{1,\infty}(\R)}$, such that for all $\ep\in(0,1)$, we have
\[\big|u^\ep(t;c)-\overline{u}(t;c)\big|\leq \frac{Ct}{|\log \ep|},\quad t>\ep|\log \ep|,\]
and
\[\big|u^\ep(t;c)-\overline{u}(t;c)\big|\leq \min\Big\{C\ep,2\Big(\max_{x\in \mathcal F}|F'(x)|+\|A'\|_\infty\Big)t\Big\},\quad t\in[0,\ep|\log \ep|].\]
\end{prop}
\begin{proof}
Since $F$ is convex, $-F'$ is non-increasing. It follows that $\bar f$ is non-increasing. Repeating the proof of \eqref{modu}, we get \eqref{modu00}. It suffices to check that condition (f2) is satisfied. Noticing that \eqref{FA} is a gradient system, define $\Psi(x)=F(x)+\ep A(\frac{x}{\ep})$, one can easily check that $\frac{d}{dt}\Psi(u^\ep(t;c))\leq 0$. Therefore, for $\ep\in (0,1)$, 
\[F(u^\ep(t;c))\leq F(c)+\ep A\Big(\frac{c}{\ep}\Big)-\ep A\Big(\frac{u^\ep(t;c)}{\ep}\Big)\leq \sup_{c\in\mathcal K}F(c)+2\|A\|_\infty,\]
which implies that $u^\ep(t;c)\in\mathcal F$ for all $t>0$. Hence, we can modify the function $F'$ outside $\mathcal F$, to get a bounded function $\hat f$, without change of solutions. Since $F$ is convex, $F'$ is a non-decreasing function. Let $x_1$ (resp. $x_2$) be the left end point (resp. the right end point) of the compact set $\mathcal F$. Define
\[\hat f(x)=F'(x)\ \text{on }\mathcal F,\quad \hat f(x)=F'(x_1)\ \text{for }x<x_1,\quad \hat f(x)=F'(x_2)\ \text{for }x>x_2.\]
Then the function $-\hat f(u)-A'(r)$ satisfies (f2), and
\[\dot u^\ep(t;c)=-\hat f(u^\ep(t;c))-A'\Big(\frac{u^\ep(t;c)}{\ep}\Big).\]
It remains to check whether the limit equation is unchanged. We recall that by definition, the value of $\bar f$ remains unchanged on $\mathcal F$. Since $u^\ep(t;c)\in \mathcal F$, the limit $\ol u(t;c)\in\mathcal F$. So $\ol u(t;c)$ solves $\dot{\ol u}=\bar f(\ol u)$. Then we apply Theorem \ref{thm2} to complete the proof.
\end{proof}

\section*{Acknowledgements}

The author is grateful to Professor Hung V. Tran for suggesting the problem and providing valuable guidance throughout the project, and to Professor Hiroyoshi Mitake for his insightful comments on the manuscript. The author also thanks Son N.T. Tu for suggesting the use of \cite{TZZ}, which led to an improvement of the regularity assumption in Theorem \ref{qp}, and the anonymous referees for their careful reading and helpful comments. The author was supported by the JSPS grant: KAKENHI \#26KF0103 and partially supported by the National Natural Science Foundation of China (Grant No. 12571197).

\section*{Declarations}

\noindent {\bf Conflict of interest statement:} The author states that there is no conflict of interest.

\medskip

\noindent {\bf Data availability statement:} Data sharing not applicable to this article as no datasets were generated or analyzed during the current study.

\end{document}